\documentclass[A4paper, 12pt]{article}

\usepackage{hyperref}
\usepackage[utf8]{inputenc}
\usepackage[OT2,T1]{fontenc}
\usepackage[english,french]{babel}
\usepackage[osf]{mathpazo}

\usepackage{amssymb}
\usepackage{authblk}

\DeclareSymbolFont{cyrletters}{OT2}{wncyr}{m}{n}
\DeclareMathSymbol{\Sha}{\mathalpha}{cyrletters}{"58}
\DeclareMathSymbol{\rde}{\mathalpha}{cyrletters}{"64}

\def\cC{{\mathcal C}}

\def\cO{{\mathcal O}}

\def\ord{{\rm ord}}

\def\N{{\bf N}}

\def\R{{\bf R}}
\def\T{{\bf T}}
\def\T{{\bf T}}

\def\bull{\vrule height .9ex width .8ex depth -.1ex }

\def\textem#1{{\em #1\/}}

\long\def\textem#1{{\em #1\/}}

\newcounter{thenum}
\def\texttheo{\relax}
\newenvironment{theorem}{\medbreak\refstepcounter{thenum}
\noindent\textsc{Theorem} %
\thethenum. \texttheo ---  \it  }{\rm }
\newenvironment{e-proposition}{\medbreak\refstepcounter{thenum}
\noindent\textsc{Proposition} \thethenum. ---  \it  }{\rm }

\newenvironment{e-definition}{\medbreak\refstepcounter{thenum}
\noindent\textsc{Definition} \thethenum. ---  \it  }{\rm }
\newenvironment{lemma}{\medbreak\refstepcounter{thenum}\noindent{\it Lemma} %
\thethenum. ---  \it  }{\rm }

\newenvironment{e-rem}{\medbreak\refstepcounter{thenum}{}%
 \thethenum) }{}

\newenvironment{e-ex}{\medbreak\refstepcounter{thenum}{}%
 \thethenum) }{}
\newenvironment{proof}{\smallbreak\noindent{\sc Proof.} --- \rm}{\quad\bull\smallskip\rm}

\newenvironment{theoreme}{\medbreak\refstepcounter{thenum}\noindent\textsc{Théorème} %
\thethenum. \texttheo---  \it  }{\rm }

\newenvironment{definition}{\medbreak\refstepcounter{thenum}
\noindent\textsc{Définition} \thethenum. ---  \it  }{\rm }

\newenvironment{exemple}{\medbreak\refstepcounter{thenum}\noindent{\sl
    Exemple} %
\thethenum. --- }{}

\newenvironment{preuve}{\smallbreak\noindent{\sc Preuve.} ---
  \rm}{\quad\bull\smallskip\rm} 

\begin{document}

\begin{center}{\LARGE\parindent=0pt\selectlanguage{french} 
Un sous-système d'un système plat
  de dimension différentielle $2$ est plat
\vskip 1cm

\selectlanguage{english}
A subsystem of a flat system of differential dimension $2$ is flat.

}
\end{center}
\vskip2cm

\hbox to \hsize{\parindent =0pt\hbox to 2.5cm{\hfill}\hss\vbox{\hsize = 7cm {\large François \textsc{Ollivier}}
\bigskip

LIX, UMR CNRS 7161 

École polytechnique 

91128 Palaiseau \textsc{cedex}

France

\smallskip

{\scriptsize francois.ollivier@lix.polytechnique.fr}
} \hss\hss
\vbox{\hsize = 7cm {\large Brahim \textsc{Sadik}}
\bigskip

Département de Mathématiques 

Faculté des Sciences Semlalia 

B.P. 2390 Marrakech 

Maroc
\smallskip

{\scriptsize sadik@ucam.ac.ma}}\hss}
\vskip 0.3cm

\begin{center}\parindent =0pt Mai 2017
\end{center}
\vfill

{\small
\hbox to \hsize{\hss\vbox{\hsize= 6.2cm\selectlanguage{english}
\noindent \textbf{Abstract.} 
A subsystem of a flat system of
  differential dimension at most $2$ is flat. Furthermore, if such a flat
  system is stationary, we show that there exist
  flat outputs not depending on the time.
}
\hss\hss

\vbox{\hsize= 6.2cm\selectlanguage{french}
\noindent \textbf{Résumé.} 
Un sous-système de dimension différentielle
au plus $2$ d'une extension plate est plate. Si un tel système
plat est stationnaire, il
admet des sorties plates indépendantes du temps.
}\hss}}
\vfill\eject

\hbox to 0pt{}\vskip -1truecm
\selectlanguage{english}
\section*{Abridged English version}

\noindent The result that will be proved corresponds in control theory
to the fact that a system with $2$ controls which is linearizable by
exogenous feedback is linearizable by endogenous feedback or also that
any \textem{subsystem} of a \textem{flat system} is flat. We
refer to \cite{fliess,fliess3,Sira-Ramirez2004,Levine2009} for more details on flat
systems, a notion that goes back to Monge's problem
\cite{Monge1787,Hilbert1912,Cartan1915,Zervos1932}, and to
\cite{vinogradov,zharinov,fliess2} for the formalism of diffiety theory. We
will use 
Rouchon's lemma \cite{Rouchon1994,Ollivier1998}, following the
notations and definitions of \cite{OllivierSadik2006a}. A
\textem{diffiety extension}, or \textem{system}, $V/U$ will be a
diffiety with a projection $V\mapsto U$ that is a diffiety morphism
and a \textem{subsystem} $W/U$ of $V/U$ is such that
$\cO(W)\subset\cO(V)$. For brevity, we denote $\partial/\partial x$ by
$\partial_{x}$. 
\bigskip

A \textem{system of differential dimension} $m$
$V/U_{\delta_{0}}$ is an open subset of
$\R^{n}\times\left(\R^{\N}\right)^{m}\times U$ with a derivation
$\delta_{0}+\sum_{i=1}^{n}f_{i}(x,z,u)\partial_{x_{i}}+\sum_{j=1}^{m}\sum_{k=0}^{\infty}
z_{j}^{(k+1)}\partial_{z_{j}^{(k)}}$.

The \textem{trivial system of differential dimension $m$},
$T^{m}/U_{\delta_{0}}$, is $\left(\R^{\N}\right)^{m}\times U$ equipped
with the derivation $\delta_{0}+\sum_{j=1}^{m}\sum_{k=0}^{\infty}
z_{j}^{(k+1)}\partial_{z_{j}^{(k)}}$. We may then define flat systems.
\medskip

\begin{e-definition}\label{e-flatness} 
A system $V/U_{\ast}$ is \textem{parametrizable} if there exist a
diffiety extension $U/U_{\ast}$ and $\phi:W/U\mapsto
V\times_{U_{\ast}}U$ a morphism of extensions of $U$, where $W$ is an
open subset of $T^{\mu}/U$, such that $\mathrm{Im}\phi$ is dense in
$V\times_{U_{\ast}}U$.
\smallskip

A system $V/U$ is \textem{flat} if there exists a dense
  open set $W\subset V$ such that any $x\in W$ admits a neighbourhood
  $O$ isomorphic by $\phi: O\mapsto \tilde O$ to an open subset $\tilde
  O$ of $T^{m}\times U$; the functions $\phi^{\ast}(z_{j})$
  that generate $\cO(O)$ are called \textem{flat outputs}.
\end{e-definition}
\bigskip

By convention, $\ord_{z}A=-\infty$ if $A$ is free from $z$ and its
derivatives. In the case of a parametrizable diffiety, we identify a
function $x\in\cO(V)$ with the function
$x(z)=\phi^{\ast}(x)\in\cO(W)$.

Our theorem may be stated as follows. The
case $m=1$ is classical \cite{Charlet1989}. The proof will make a
repeated use of Rouchon's lemma \cite{Rouchon1994,Ollivier1998}, given
below.
\smallskip

\begin{theorem}\label{e-endo}
A parametrizable extension of differential dimension at most $2$
is flat.
\end{theorem}
\medskip

\begin{theorem}\label{e-Rouchon}
Let $V/U_{\ast}$ be a parametrizable diffiety extension,
$x_{i}\in\cO(V)$, $i=1\ldots n$ a family of nonconstant functions on
$V$, $H(x)=0$ a differential equation satisfied by the $x_{i}$ and
$e_{i}:=\ord_{x_{i}}H$.

Using the notations of def.~\ref{e-flatness}, let $z_{j}$ be
coordinates on $T^{\mu}$ and assume the $\max_{i=1}^{n}\ord_{z_{1}}
x_{i}=r>-\infty$. If $e_{i}=0$ $\Rightarrow$ $\ord_{z_{1}}
x_{i_{0}}<r$, then using the notation
$D:=\sum_{e_{i}>0}C_{i}\partial_{x_{i}^{(e_{i})}}$ where the $C_{i}$
are new variables with $DC_{1}=0$, the $n$-tuple $(\partial_{z_{1}^{(r)}}
x_{i}^{(e_{i})})$ is a solution of the equations
$\phi^{\ast}(D^{k}P)=0$.
\end{theorem}
\begin{proof}
It suffices to remark that,
$\partial_{z_{1}^{(r+1)}}x_{i}^{(e_{i})}=\partial_{z_{1}^{(r)}}x_{i}$
for $e_{i}>0$,  so that
$\partial_{z_{1}^{(r+1)}}^{2}x_{i}=0$.
Substituting $\partial_{z_{1}^{(r+1)}}x_i$ to $C_{i}$ in $D^{k}P$, one
gets
$\partial_{z_{1}^{(r+1)}}^{k}P=0$.
\end{proof}
\medskip

From now on, we only consider differential dimension $2$.
We may complete some local coordinates on $V$, considered as functions
of the $z_{j}$ using the morphism $\phi$, to get local coordinates on
$W$, say by choosing $z_{i}, \ldots, z_{i}^{(s_{i})}$, $i=1,2$,
($s_{1}\le r$), and the $z_{j}^{(k)}$, $j>2$, $k\in\N$. We may then
express the derivations $\partial/\partial z_{1}^{(k)}$ in the $V$
coordinates. In this setting, for state equations defining $V$,
Rouchon's lemma means that $\partial_{z_{1}^{(r+1)}}$ and
$\partial_{z_{1}^{(r)}}=[\delta,\partial_{z_{1}^{(r+1)}}]$ do
commute. The next lemma goes one step further by considering 
$\partial_{z_{1}^{(r-1)}}=[\delta,\partial_{z_{1}^{(r)}}]$
(\textit{cf.} \cite{zharinov1996}).

\begin{lemma}\label{e-Rouchon2} Under the hypotheses of th.~\ref{e-Rouchon}, assume that
  the system $V/U$ is locally defined by explicit equations of order
  $1$.  $P_{i}:=x_{i}'-h_{i}(x,x_{1}',x_{2}',u)=0$, $2< i\le n$.

\noindent \textem{a)} If $V/U$ is parametrizable, the homogeneous
  ideal $(D^{k}P_{i}|1\le i\le n-2,\> k\in\N)$ is of projective dimension $0$,
  iff
  at least one of the equations $P_{i}$ is non linear in the derivatives
  $x_{i}'$.

\noindent \textem{b)} In this case, the state equations can be rewritten
$x_{i}'=f_{i}(x,v_{2},u)v_{1}+g_{i}(x,v_{2},u)$, where the $v_{i}$ are functions of
the $x_{i}$ and the $x_{j}'$, $j=1,2$, such that
$\partial_{z_{1}^{(r+1)}}v_{2}=0$, $f_{1}=1$ (\textem{i.e.}
$v_{1}=x_{1}'$) and $f_{2}=v_{2}$ if $f_{2}$ depends on $v_{2}$.

\noindent \textem{c)} We have moreover $\partial_{z_{1}^{(r)}}f_{i}(x,v_{2},u)=0$.
\end{lemma}
\begin{proof} \textem{a)} --- By th.~\ref{e-Rouchon}, the dimension is at least
  $1$. Now, if $P_{i}$ is non linear in $x_{i}'$, $i=1,2$, then a non
  trivial relation $D^{2}P_{i}=0$, so that the dimension is at most
  $1$. \textem{b)} --- Up to a permutation of indices, we may assume that
  $C_{1}\neq0$. Now, $C_{2}=F(x,x_{1}',x_{2}')C_{1}$ with
  $\partial_{z_{1}^{(r)}}F=0$. We take $v_{1}=x_{1}'$. If $F$ depends
  on $x_{2}'$, we choose $v_{2}=F$, $f_{2}(v_{2})=v_{2}$ and $g_{2}=x_{2}'-v_{2}$ or
  else $f_{2}(x)=F(x)$, $v_{2}=x_{2}'-f_{2}(x)v_{1}$
  and $g_{2}(v_{2})=v_{2}$. \textem{c)} --- We have
$\partial_{z_{1}^{(r+1)}}=A\partial_{v_{1}}+2A'\partial_{v_{1}'}+\cdots+B\partial
v_{2}'+\cdots$\ 
So
$\partial_{z_{1}^{(r)}}=\sum_{i=1}^{n}Af_{i}\partial_{x_{i}}+
A'\partial_{v_{1}}\cdots+B\partial_{v_{2}}+\cdots$  
Going one step further, we get, as terms in $A'$ cancel: 
$\partial_{z_{1}^{(r-1)}}=\sum_{i=1}^{n}\left(Af_{i}'+
\hbox{terms of order at most $r$ in $z_{1}$}\right)\partial_{x_{i}}+\cdots$
As $[\partial_{z_{1}^{(r+1)}},\partial_{z_{1}^{(r-1)}}]=0$, we need
  have $\partial_{z_{1}^{(r+1)}}f_{i}'=\partial_{z_{1}^{(r)}}f_{i}=0$. 
\end{proof}
\medskip

\noindent\textsc{Sketch of the proof of th.~\ref{e-endo}.} --- 
Denote by $z$, $v$, $u$ $u_{\ast}$ coordinate functions on $T^{\mu}$,
$V$, $U$, $U_{\ast}$. If the result is false, there exists an open
subset $O=\phi(\tilde O)\subset V$ such that no open subset
$O_{\prime}=\phi(\tilde O_{\prime})\subset O$ is isomorphic to an open
subset of $T^{m}$. We will look for a contradiction.
\smallskip

In the case $m=1$, we choose
$x_{i}$, $1\le i\le n$ such that the $x_{i}$ and their derivatives are
local coordinates of $O_{\prime}\subset O$. We assume moreover that
these function satisfy equations of order $1$
$x_{i}'=h_{i}(x, x_{1}', u_{\ast})$, $1<i\le n$ and that $n$ is
minimal. If $n=1$, $O_{\prime}$ is flat. 

If not, th.~\ref{e-Rouchon} implies that the $h_{i}$ are linear in
$x_{1}'$: $h_{i}=f_{i}x_{i}'+g_{i}$. Then, we may replace the $x_{i}$
by $n-1$ independent solutions $y(x,u_{\ast})$ of the differential
system $\partial_{x_{1}} Y+\sum_{i=2}^{n}f_{i}\partial_{x_{i}}
Y=0$. The derivatives $y_{i}'$ do not depend on $x_{1}'$, and
the $y_{i}$ satisfy a new system of order $1$ contradicting the
minimality of $n$.
\smallskip

In the case $m=2$, we also consider $x_{i}(z,u,u_{\ast})$, $1\le i\le n$,
that satisfy a system of order $1$ $x_{i}'=h_{i}(x, x_{1}',
x_{2}',u_{\ast})$, $2<i\le n$ and 
$\max_{i=1}^{n}\ord_{z_{1}} x_{i}=r>-\infty$. We assume that the
couple $(r, n)$ is minimal for lexicographic ordering. 

If $n=2$, $O_{\prime}$ is flat. If the $h_{i}$ do not depend on
$x_{1}'$ or $x_{2}'$, we are reduced to the case $m=1$, already
considered. We distinguish two cases.
\smallskip

i) If the $h_{i}$ are all linear in $x_{1}'$ and $x_{2}'$: 
$h_{i}=\sum_{j=1}^{2}f_{i,j}x_{j}'+g_{i}$,
$1\le i\le n$,  we
replace the $x_{i}$ by $n-1$ independent solutions $y(x, u_{\ast})$,
of the differential equation $\partial_{x_{1}}
Y+\sum_{i=3}^{n}f_{i,1}\partial_{x_{i}} Y=0$. The order
$\max_{i=1}^{n-1}\ord_{z_{1}} y_{i}$ is at most $r$. The $y_{i}'$ do
not depend on $x_{1}'$, so that they satisfy a system of order
$1$. If the $y_{i}'$ depend on $x_{1}$, they satisfy our hypotheses,
which contradicts the minimality of $n$. If not, we are reduced to the
case $m=1$ and we just have to complete the flat output for the
diffiety defined by the $y_{i}$ with $x_{1}$ to conclude.
\smallskip

ii) If the $h_{i}$ are not linear in $x_{1}'$ and $x_{2}'$, by
\textem{lem.}~\ref{e-Rouchon2} the state equations can be rewritten
$x_{i}'=f_{i}(x,v_{2},u_{\ast})*v_{1}+g_{i}(x,v_{2},u_{\ast})$, with
$\partial_{z_{1}^{(r+1)}}v_{2}=0$ and
$\partial_{z_{1}^{(r)}}f_{i}(x,v_{2},u)=0$.
We can replace the $x_{i}$ by 
$n-1$ independent solutions $y_{i}(x,v_{2},u_{\ast})$  of the
differential equation $\sum_{i=1}^{n}f_{i}\partial_{x_{i}}Y=0$,
completed with $v_{2}$ if the $f_{i}$ depend on $v_{2}$. The
$y_{i}'$ must depend on $x_{1}$; if not the $f_{i}$ must be constants,
the $y_{i}'$ satisfy linear equations so that the $h_{i}$ should have been
linear in the $x_{j}'$, $j=1,2$.

So, the $y_{i}$ (and $v_{2}$ if $f_{2}=v_{2}$) must satisfy a system of
order $1$ and, according to \textem{lem.}~\ref{e-Rouchon2} are of
order less than $r$ in $z_{1}$. A final contradiction that concludes
the proof.
\smallskip

\textit{E.g.}, the diffieties $U$ and $U_{\ast}$ may be respectively
$\R$, standing for the time variable $t$ with a derivation
$\partial/\partial t$ and  a single point with derivation $0$, if $V$
is associated to a stationary model. Then, the theorem asserts that
$V$ is flat with flat outputs not depending on the time, answering a
problem raised by Pereira Da Silva and Rouchon \cite{PereiraRouchon2004}. 

\selectlanguage{french}

\section*{Introduction}
\label{intro}

\noindent Si la notion mathématique remonte aux travaux de Monge
\cite{Monge1787} et a été étudiée au début du \textsc{xx}$^{\rm
  e}$~siècle par Hilbert \cite{Hilbert1912}, Cartan \cite{Cartan1915}
ou Zervos \cite{Zervos1932}, les systèmes plats ont été inventés sous
ce nom pour les besoins de l'automatique
\cite{fliess,fliess3,Levine2009}. Dans ce cadre, le résultat qui va
être prouvé en dimension différentielle au plus $2$, signifie qu'un
système linéarisable par bouclage exogène est linéarisable par
bouclage endogène. Sommairement, \textem{tout sous-système d'un
système pla est plat}, c'est-à-dire que si les solutions d'un
système d'EDO sont paramétrables par $m$ fonctions arbitraires, il
existe un tel paramétrage localement bijectif. Le cas $m=1$ est une
conséquence des résultat de Charlet \textit{et al.} \cite{Charlet1989}
ou dans le cas d'un paramétrage rationnel du théorème de Lüroth--Ritt,
mais qui n'a pas d'analogue en dimension différentielle
$2$ \cite{Ollivier1998}.

La définition de la platitude peut varier selon que l'on impose ou non
à un système stationnaire de posséder un paramétrage indépendant du
temps (\textit{cf.} Pereira da Silva et Rouchon
\cite{PereiraRouchon2004}). On montrera que, en dimension
différentielle au plus $2$, ces deux définitions coïncident,
c'est-à-dire que \textem{si un système stationnaire possède des sorties plates
dépendant du temps, il en existe d'autres indépendantes}.

\section{Diffiétés plates}

\noindent 
Pour la notion de diffiété \cite{vinogradov,zharinov}, nous adoptons les
conventions de \cite{OllivierSadik2006a}.
Soit $I$ un ensemble dénombrable, on appellera
\textem{diffiété} un ouvert $V$ de $\R^{I}$ pour la topologie la
plus grossière rendant pour tout $i_{0}\in I$ les projections
$\pi_{i_{0}}:(x_{i})_{i\in I}\mapsto x_{i_{0}}$ continues, muni d'une
dérivation $\delta=\sum_{i\in I} c_{i}(x){\partial\over\partial
x_{i}}$, où les $c_{i}$ appartiennent
à $\cO(V)$, l'anneau des applications $\cC^{\infty}$ de $V$ dans
$\R$ ne dépendant que d'un nombre \textem{fini} de
coordonnées. Par concision, $\partial/\partial x$ sera noté $\partial_{x}$.

Un \textem{morphisme de diffiétés} est une application
$\phi:V_{\delta_{1}}\mapsto V_{\delta_{2}}$, définie par des
fonctions $\cO(V)$ et telle que
$\delta_{1}\circ\phi^{\ast}=\phi^{\ast}\circ\delta_{2}$, où
$\phi^{\ast}:\cO(V_{2})\mapsto\cO(V_{1})$ est l'application duale de
$\phi$.

Une \textem{extension de diffiétés}, ou un \textem{système}, noté
$V/U$, est un couple de diffiétés muni d'une projection $\pi:V\mapsto
U$ surjective qui est un morphisme de diffiétés. Il s'agit donc d'un
fibré sur $U$, avec une projection compatible avec la structure de
diffiété.  Un morphime d'extensions $\phi:V_{1}/U\mapsto V_{2}/U$ est
un morphisme de $V_{1}$ dans $V_{2}$ tel que
$\pi_{2}\circ\phi=\pi_{1}$.  Un \textem{sous-système} $W/U$ de $V/U$
est une extension de $U$ telle que $\cO(W)\subset\cO(V)$. Soient $V/U$
et $W/U$ deux extensions de diffiétés, leur produit fibré est muni
d'une structure naturelle d'extension de $U$ (ainsi que de $V$ ou de
$W$), notée $V\times_{U} W$.

Un \textem{système de dimension différentielle} $m$
$V/U_{\delta_{0}}$ est un ouvert de
$\R^{n}\times\left(\R^{\N}\right)^{m}\times U$ muni d'une dérivation
de la forme $\delta_{0}+\sum_{i=1}^{n}f_{i}(x,z)\partial_{x_{i}}+\sum_{j=1}^{m}\sum_{k=0}^{\infty} 
z_{j}^{(k+1)}\partial_{z_{j}^{(k)}}$. 
\textem{L'extension triviale} de dimension différentielle $m$, que
l'on note $T^{m}/U_{\delta_{0}}$, est $\left(\R^{\N}\right)^{m}\times U$ muni
de la dérivation $$\delta_{0}+\sum_{i=1}^{m}\sum_{k=0}^{\infty}
z_{i}^{(k+1)}\partial_{z_{i}^{(k)}}.$$ 
\smallskip

\begin{definition}\label{platitude}
Un système $V/U_{\ast}$ est \textem{paramétrable} s'il existe
un système $U/U_{\ast}$ et $\phi:W/U\mapsto
V\times_{U_{\ast}}U$ un morphisme d'extension de $U$, où $W$  est un ouvert
de $T^{\mu}/U$, tel que $\mathrm{Im}\phi$ est dense dans
$V\times_{U_{\ast}}U$.

Un système $V/U$ sera dit \textem{plat} s'il existe un ouvert dense $W$ de $V$
tel que tout $x$ appartenant à $W$ admette un voisinage $O$ isomorphe
par $\phi: O\mapsto \tilde O$ à un
ouvert $\tilde O$ de \textem{l'extension triviale} $\T^{m}\times U$. Les
fonctions ${\phi^{\ast}}(z_{i})$, où les $z_{i}$ définissent
l'extension triviale, sont appelées \textem{sorties plates}.
\end{definition}
\smallskip

\begin{theoreme} \textem{(Endogène=exogène)}\label{endogene=exogene}
Une extension paramétrable de dimension différentielle au plus $2$
est plate.
\end{theoreme}
\bigskip

Les diffiétés $U_{\ast}$ et $U$ (déf.~\ref{platitude}) peuvent, par
exemple, faire intervenir une variable {\og}temps{\fg} $t$ avec
$t'=0$: si celui-ci n'apparaît pas dans $U_{\ast}$, il est absent des
sorties plates de $V/U_{\ast}$.

\section{Lemme de Rouchon et itération}

\noindent Par convention, $\ord_{z}A=-\infty$ si $A$ est indépendent
de $z$ et de ses dérivées.derivatives. Pour une diffiété paramétrable,
on identifiera la fonction $x\in\cO(V)$ avec la fonction
$x(z)=\phi^{\ast}(x)\in\cO(W)$.

\begin{theorem}\label{Rouchon}
Soit $V/U_{\ast}$ un système paramétrable,
$x_{i}\in\cO(V)$, $i=1\ldots n$ une famille de fonctions non constantes sur
$V$, $H(x)=0$ une équation différentielle satisfaite par les $x_{i}$ et
$e_{i}:=\ord_{x_{i}}H$.

Avec les notations de la déf.~\ref{platitude}, soient $z_{j}$ des
coordonnées sur $T^{\mu}$ telles que $\max_{i=1}^{n}\ord_{z_{1}}
x_{i}=r>-\infty$. Si $e_{i}=0$ $\Rightarrow$ $\ord_{z_{1}}
x_{i_{0}}<r$, alors notant
$D:=\sum_{e_{i}>0}C_{i}\partial_{x_{i}^{(e_{i})}}$ où les $C_{i}$
sont de nouvelles variables avec $DC_{1}=0$, le $n$-uplet $(\partial_{z_{1}^{(r)}}
x_{i}^{(e_{i})})$ est solution des équations
$\phi^{\ast}(D^{k}P)=0$.
\end{theorem}
\begin{preuve}
Il suffit de remarquer que
$\partial_{z_{1}^{(r+1)}}x_{i}^{(e_{i})}=\partial_{z_{1}^{(r)}}x_{i}$
pour $e_{i}>0$, de sorte que
$\partial_{z_{1}^{(r+1)}}^{2}x_{i}=0$.
Substituant $\partial_{z_{1}^{(r+1)}}x_i$ to $C_{i}$ dans $D^{k}P$, on
obtient $\partial_{z_{1}^{(r+1)}}^{k}P=0$.
\end{preuve}
\medskip

Nous nous limitons maintenant à la dimension différentielle $2$.
On peut compléter des coordonnées $x_{i}$ sur $V$, considérées comme
des fonctions des $z_{j}$ grâce au morphisme $\phi$, pour obtenir des
coordonnées locales sur
$W$, \textit{e.g.} en choisissant $z_{i}, \ldots, z_{i}^{(s_{i})}$, $i=1,2$,
($s_{1}\le r$), et les $z_{j}^{(k)}$, $j>2$, $k\in\N$. On peut alors
exprimer les dérivations $\partial_{z_{1}^{(k)}}$ dans les coordonnées
de $V$. De la sorte, pour des équations d'état définissant $V$,
le \textit{lemme} \ref{Rouchon} signifie que $\partial_{z_{1}^{(r+1)}}$ et
$\partial_{z_{1}^{(r)}}=[\delta,\partial_{z_{1}^{(r+1)}}]$
commutent. Le lemme suivant va un cran plus loin en considérant
$\partial_{z_{1}^{(r-1)}}=[\delta,\partial_{z_{1}^{(r)}}]$
(\textit{cf.} \cite{zharinov1996}).

\begin{lemma}\label{Rouchon2} Sous les hypothèses du
  th.~\ref{Rouchon}, supposons que
  le système $V/U$ est localement défini par des équations explicites
  d'ordre 
  $1$:  
\begin{equation}\label{state-eq}
P_{i}:=x_{i}'-h_{i}(x,x_{1}',x_{2}',u)=0,\quad 2< i\le n.
\end{equation}

\noindent \textem{a)} Si $V/U$ est paramétrable, l'idéal homogène
  $(D^{k}P_{i}|1\le i\le n-2,\> k\in\N)$ est de dimension projective $0$ ssi
  ssi l'une des équations $P_{i}$ est non linéaire en les dérivées
  $x_{i}'$.

\noindent \textem{b)} Alors, les équations (\ref{state-eq}) se réécrivent
$x_{i}'=f_{i}(x,v_{2},u)v_{1}+g_{i}(x,v_{2},u)$ $(2)$, où les $v_{i}$ sont
des fonctions des $x_{i}$ et des $x_{j}'$, $j=1,2$, telles que
$\partial_{z_{1}^{(r+1)}}v_{2}=0$, $f_{1}=1$ (\textem{i.e.}
$v_{1}=x_{1}'$) et $f_{2}=v_{2}$ si $f_{2}$ dépend de $v_{2}$.

\noindent \textem{c)} On a en outre $\partial_{z_{1}^{(r)}}f_{i}(x,v_{2},u)=0$.
\end{lemma}
\begin{preuve} \textem{a)} --- Par le th.~\ref{Rouchon}, la dimension
  est au moins $0$. Si $P_{i}$ est non linéaire en les $x_{i}'$,
  $i=1,2$, alors il existe une  équation non triviale
  $D^{2}P_{i}=0$, et la dimension est au plus $0$.
  \textem{b)} --- À permutation près, on peut supposer
  $C_{1}\neq0$. Alors, $C_{2}=F(x,x_{1}',x_{2}')C_{1}$ avec
  $\partial_{z_{1}^{(r)}}F=0$. Nous prenons $v_{1}=x_{1}'$. Si $F$ dépend
  de $x_{2}'$, on pose $v_{2}=F$, $f_{2}(v_{2})=v_{2}$ et
  $g_{2}=x_{2}'-v_{2}$ ou sinon
  $f_{2}(x)=F(x)$, $v_{2}=x_{2}'-f_{2}(x)v_{1}$
  et $g_{2}(v_{2})=v_{2}$. \textem{c)} --- Nous avons
$\partial_{z_{1}^{(r+1)}}=A\partial_{v_{1}}+2A'\partial_{v_{1}'}+\cdots+B\partial
v_{2}'+\cdots$\ 
Donc
$\partial_{z_{1}^{(r)}}=\sum_{i=1}^{n}Af_{i}\partial_{x_{i}}+
A'\partial_{v_{1}}\cdots+B\partial_{v_{2}}+\cdots$  
Au cran suivant, comme les terme en $A'$ s'annulent, on a: 
$\partial_{z_{1}^{(r-1)}}=\sum_{i=1}^{n}\left(Af_{i}'+
\hbox{des termes d'ordre au plus $r$ en $z_{1}$}\right)\partial_{x_{i}}+\cdots$
Comme $[\partial_{z_{1}^{(r+1)}},\partial_{z_{1}^{(r-1)}}]=0$, on doit
avoir, $\partial_{z_{1}^{(r+1)}}f_{i}'=\partial_{z_{1}^{(r)}}f_{i}=0$. 
\end{preuve}

\section{Mise en œuvre du lemme de Rouchon}

\noindent Soient $z$, $v$, $u$ $u_{\ast}$ des coordonnées sur
$T^{\mu}$, $V$, $U$, $U_{\ast}$. Si le théorème est faux, il existe un
ouvert $O=\phi(\tilde O)\subset V$ tel qu'aucun ouvert
$O_{\prime}=\phi(\tilde O_{\prime})\subset O$ n'est plat. Nous allons
chercher une contradiction.
\medskip

Dans le cas $m=1$, soient
$x_{i}$, $1\le i\le n$ des fonctions définissant avec leurs dérivées des
coordonnées locales de $O_{\prime}\subset O$. On suppose en outre
qu'elles satisfont des équations d'ordre $1$
$x_{i}'=h_{i}(x, x_{1}', u_{\ast})$, $1<i\le n$ et que $n$ est
minimal. Si $n=1$, $O_{\prime}$ est plat. 

Sinon, le th.~\ref{Rouchon} implique que les $h_{i}$ sont linéaires en
$x_{1}'$: $h_{i}=f_{i}x_{i}'+g_{i}$. On peut alors remplacer les $x_{i}$
par $n-1$ solutions indépendantes $y(x,u_{\ast})$ de l'équation 
$\partial_{x_{1}} Y+\sum_{i=2}^{n}f_{i}\partial_{x_{i}}
Y=0$. Les dérivées $y_{i}'$ ne dépendent pas de $x_{1}'$, de sorte que
les $y_{i}$ satisfont un nouveau système d'ordre $1$ contredisant la
minimalité de $n$.
\medskip

Dans le cas $m=2$, on considère aussi $x_{i}(z,u,u_{\ast})$, $1\le i\le n$,
qui satisfont un système d'ordre $1$ $x_{i}'=h_{i}(x, x_{1}',
x_{2}',u_{\ast})$, $2<i\le n$ and 
$\max_{i=1}^{n}\ord_{z_{1}} x_{i}=r>-\infty$. Nous supposons le couple
$(r, n)$ minimal pour l'ordre lexicographique.

Si $n=2$, $O_{\prime}$ est plat. Si les $h_{i}$ ne dépendent pas de
$x_{1}'$ ou $x_{2}'$, on se ramène au cas $m=1$, déjà traité.
On distingue deux situations.
\smallskip

i) Si les $h_{i}$ sont tous linéaires en $x_{1}'$ et $x_{2}'$: 
$h_{i}=\sum_{j=1}^{2}f_{i,j}x_{j}'+g_{i}$,
$1\le i\le n$, on remplace les $x_{i}$ par $n-1$ solutions
indépendantes $y(x, u_{\ast})$ de l'équation $\partial_{x_{1}}
Y+\sum_{i=3}^{n}f_{i,1}\partial_{x_{i}} Y=0$. L'ordre
$\max_{i=1}^{n-1}\ord_{z_{1}} y_{i}$ est au plus $r$. Les $y_{i}'$ ne
dépendent pas de $x_{1}'$, et satisfont donc un nouveau système
d'ordre $1$. Si les $y_{i}'$ dependent de $x_{1}$, ils satisfont nos
hypothèses, contredisant la minimalité de $n$. Sinon, on est ramené au cas
$m=1$ et il suffit de compléter la sortie plate de la diffiété définie
par les $y_{i}$ avec $x_{1}$ pour conclure.
\smallskip

ii) Si les $h_{i}$ ne sont pas linéaires en $x_{1}'$ et $x_{2}'$, par le
\textem{lem.}~\ref{Rouchon2} les équation d'état peuvent être réécrites
$x_{i}'=f_{i}(x,v_{2},u_{\ast})*v_{1}+g_{i}(x,v_{2},u_{\ast})$, avec
$\partial_{z_{1}^{(r+1)}}v_{2}=0$ et
$\partial_{z_{1}^{(r)}}f_{i}(x,v_{2},u)=0$.
On remplace les $x_{i}$ par 
$n-1$ solutions indépendantes $y_{i}(x,v_{2},u_{\ast})$ de l'équation
différentielle $\sum_{i=1}^{n}f_{i}\partial_{x_{i}}Y=0$,
cpomplétées avec $v_{2}$ si les $f_{i}$ dépendent de $v_{2}$. Les
$y_{i}'$ doivent dépendre de $x_{1}$; sinon les $f_{i}$ seraient des constantes,
et les $y_{i}'$ satisferaient un système linéaire de sorte que $h_{i}$
seraient aussi linéaires en les $x_{j}'$, $j=1,2$.

Donc, les $y_{i}$ (et  $v_{2}$ si $f_{2}=v_{2}$) engendrent
$\cO(O_{\prime})$, doivent satisfaire un système d'ordre $1$ et, selon
le \textem{lem.}~\ref{Rouchon2} sont d'ordre strictement inférieur à
$r$ en $z_{1}$: une contradiction finale qui achève la preuve.
\medskip

\begin{exemple} Considérons $V$ défini par un modèle de voiture
  $\theta'=u/d\tan \phi$, $x'=u\cos\theta$, $y'=u\sin\theta$ (\textit{cf.}
  \cite[(18)]{fliess}). Les diffiétés $U=U_{0}$ sont définies par
  $d'=0$. Soient les paramétrages $\phi$ exprimés par $\theta
  =\arctan(z_{2}'/z_{1}') \mp
  \arctan(z_{3}'/((z_{1}')^{2}+(z_{2}')^{2}-(z_{3}')^{2})^{1/2})+(1\mp
  1)\pi/2$, $x=z_{1}+z_{3}\sin\theta$ et $y=z_{2}-z_{3}\cos\theta$.
  Quand ${x_{3}'}^2$ tend vers ${x_{1}'}^{2}+{x_{2}'}^{2}$, les deux
  paramétrages tendent vers le lieu où $x'=y'=0$.  Une famille
  de sorties plates est donnée, \textit{e.g.}, par
  $B_{1}:=x+C\sin\theta$ et $B_{2}:=y-C\cos\theta$, $C\in\R$.

Tout choix de fonctions $B_{i}$ de $x$, $y$ et
  $\theta$ minimise $\epsilon(=2)$, mais pas nécessairement
  $e_{2}(\in\{0,1\})$. Si $e_{2}=1$, il peut être abaissé en
remplaçant $P_{1}$ par une solution de (\ref{state-eq}) (\textit{cf.} i)
\textit{supra}). Si $e_{2}=0$ et $s=1$, alors on obtient $s=0$
  en remplaçant $A_{2}$ par $P_{1}$ (\textit{cf.} ii) a) \textit{supra}).
\end{exemple}
\medskip

\begin{exemple} En prenant $U$ défini par
  $t'=1$ et en substituant $t$ à $z_{3}$ dans les formules de
  l'exemple précédent, on obtient bien de nouveaux paramétrages
  indépendants de $t$.
\end{exemple}

\section*{Conclusion}

\noindent
Nous espérons une adaptation de cette preuve dans le cadre de l'algèbre
différentielle ou en dimension différentielle supérieure à $2$, mais
de nombreuses difficultés se présentent que nous ne savons pas surmonter.

\end{document}